\documentclass[12pt]{amsart}
\usepackage[a4paper,margin=25mm,top=27mm]{geometry}
\usepackage{amsmath,amssymb,mathtools,amsthm,stmaryrd,mathrsfs}
\usepackage[T1]{fontenc}
\usepackage{newtxtext,newtxmath}
\usepackage[colorlinks=true,linkcolor=teal,citecolor=magenta,urlcolor=blue]{hyperref}
\usepackage{enumerate}

\newtheorem{theorem}{Theorem}[section]
\newtheorem{lemma}[theorem]{Lemma}
\newtheorem{proposition}[theorem]{Proposition}
\newtheorem{corollary}[theorem]{Corollary}
\theoremstyle{definition}

\newtheorem{remark}[theorem]{Remark}

\DeclareMathOperator{\Tr}{Tr}
\DeclareMathOperator{\End}{End}
\DeclareMathOperator{\Aut}{Aut}
\DeclareMathOperator{\Hom}{Hom}
\DeclareMathOperator{\Ext}{Ext}
\DeclareMathOperator{\Der}{Der}
\DeclareMathOperator{\GL}{GL}

\DeclareMathOperator{\id}{id}
\DeclareMathOperator{\rk}{rk}
\DeclareMathOperator{\Span}{span}

\newcommand{\Sym}{\mathrm{Sym}}
\newcommand{\Harr}{\mathrm{Harr}}
\newcommand{\Mat}{\mathrm{Mat}}

\newcommand{\Alg}{\mathsf{Alg}}
\newcommand{\Leib}{\mathsf{Leib}}
\newcommand{\Lie}{\mathsf{Lie}}

\newcommand{\As}{\mathsf{As}}
\newcommand{\Comm}{\mathsf{Comm}}

\newcommand{\cO}{\mathcal{O}}

\newcommand{\kk}{\Bbbk}

\numberwithin{equation}{section}

\title{Geometry of Deformations via Incidence Varieties}
\author{Atabey Kaygun}
\email{kaygun@itu.edu.tr}
\address{Istanbul Technical University, Istanbul, Turkey}

\begin{document}

\begin{abstract}
  We provide a unified geometric realization of the classical deformation complexes.  We construct
  $\GL$-equivariant bilinear incidence varieties whose diagonal slices recover the varieties of
  associative, commutative, Leibniz, and Lie algebra structures on a finite-dimensional vector
  space. We prove that the fiber of the incidence map at a given algebra law is canonically
  isomorphic to the space of $2$-cocycles in the corresponding cohomology theory (Hochschild,
  Harrison, Leibniz, or Chevalley--Eilenberg). Furthermore, we introduce invariant bilinear forms
  to define open strata of separable and semisimple algebras, and demonstrate that these strata
  consist of open $\GL$-orbits, establishing the rigidity of generic points in the coarse moduli
  spaces for all four geometries.
\end{abstract}

\maketitle

\section*{Introduction}

We provide a new and unified geometric framework for understanding the deformation theory of
associative, Lie, and Leibniz algebras. We operate in the affine space
$V^{2,1} = \Hom(V \otimes V, V)$ of all bilinear multiplications on an $n$-dimensional $\kk$-vector
space $V$ over a field $\kk$ of characteristic 0. While the varieties $\mathrm{Alg}(V)$,
$\mathrm{Lie}(V)$, and $\mathrm{Leib}(V)$ are extensively studied within this space, we realize
these varieties as diagonal slices within bilinear incidence varieties that geometrically realize
the classical deformation complexes.

\subsubsection*{ The Associative Case}

The deformation theory of associative algebras is governed by Hochschild
cohomology~\cite{Gerstenhaber1964,NijenhuisRichardson1967}. The tangent space to the associative
variety $\mathrm{Alg}(V)$ at a point $x$ is classically identified with the space of Hochschild
$2$-cocycles, $Z^2(V_x, V_x)$, and the corresponding second cohomology group $HH^2(V_x, V_x)$
controls the moduli space and rigidity~\cite{GeissdelaPena1995}.  We provide a geometric origin for
this correspondence. We construct a $\GL(V)$-equivariant bilinear incidence variety
$\mathrm{As}(V) \subset V^{2,1} \times V^{2,1}$. We prove that the associative locus
$\mathrm{Alg}(V)$ cut out by $n^4$ quadratic equations (Proposition~\ref{prop:assoc-quadrics}) is
the diagonal slice of this larger variety (Propositions~\ref{prop:diag-equals-assoc}
and~\ref{prop:Q-rep}). We then establish that for any associative algebra $V_x$, the fiber $F_x$ of
our incidence variety is canonically isomorphic to the space of Hochschild $2$-cocycles,
$F_x \cong Z^2(V_x, V_x)$ (Proposition~\ref{prop:fiber-Z2}). This immediately and geometrically
recovers the classical result $T_x\mathrm{Alg}(V) \cong Z^2(V_x, V_x)$.

We then use this construction to study rigidity. Aguiar identified non-degeneracy of the trace form
with separability~\cite{Aguiar1998} which we use in Theorem~\ref{thm:Aguiar}. We identify the
separable locus as an open subset as $U \subset \mathrm{Alg}(V)$ using the discriminant $\Delta(x)$
of the principal trace form $T_{V_x}$, and provide an explicit formula for the dimension of the
tangent space on this locus (Theorem~\ref{thm:tangent-generic}). We then prove that the classical
cohomological condition for rigidity, $HH^2(V_x, V_x) = 0$, has a direct geometric consequence: the
isomorphism class $[V_x]$ is isolated in the coarse GIT quotient $\mathrm{Alg}(V)\sslash \GL(V)$
(Theorem~\ref{thm:rigidity-moduli}). We also show that our incidence variety admits compatible
stratifications which are indexed by Wedderburn block profiles on the separable locus
(Theorem~\ref{thm:diagonal-stratification}).

\subsubsection*{The Commutative Case}

The moduli problem in the commutative case is a refinement governed by Harrison
cohomology~\cite{Harrison1962,GerstenhaberSchack1988} or André–Quillen
cohomology~\cite{Andre1967,Quillen1970}). Our incidence construction respects this refinement. When
restricted to the commutative sublocus $\mathrm{Comm}(V)$, our incidence fibers naturally recover
the Harrison-theoretic description in terms of symmetric Hochschild cocycles. We study the open
subset $U_{\Comm}$ of finite étale algebras (cf.~\cite{Poonen2008}) defined by the non-vanishing
discriminant. We prove that this entire locus is rigid
($T_x\mathrm{Comm}(V) = T_x(\GL(V) \cdot x)$) and decomposes into a finite union of open
$\GL(V)$-orbits (Theorem~\ref{thm:rigid-UComm}). Over an algebraically closed field, this
simplifies to a single open orbit corresponding to $\kk^n$ (Corollary~\ref{cor:UComm-orbits}).

\subsubsection*{The Leibniz and Lie Cases}

We also demonstrate that our approach extends to the parallel deformation theories of Leibniz and
Lie algebras, which the literature identifies as being governed by Leibniz
cohomology~\cite{LodayPirashvili1993} and Chevalley–Eilenberg
cohomology~\cite{NijenhuisRichardson1967}, respectively.

For the right Leibniz locus $\mathrm{Leib}(V)$, also defined by $n^4$ quadratic equations
(Proposition~\ref{prop:Leibniz-quadrics}), we construct a different incidence locus defined by the
symmetrization of a bilinear map $B$. We prove that the tangent space $T_x\mathrm{Leib}(V)$ is
exactly this incidence fiber and is canonically isomorphic to the space of Leibniz $2$-cocycles,
$Z^2_{\Leib}(V_x, V_x)$ (Proposition~\ref{prop:Leib-tangent=Z2}). This provides a geometric
recovery of the full Leibniz deformation complex.

When we restrict this Leibniz incidence construction to the Lie subvariety $\mathrm{Lie}(V)$, it
specializes exactly to the Chevalley–Eilenberg complex, as shown in
Proposition~\ref{prop:tangent-Lie}. Our incidence fibers for $x \in \mathrm{Lie}(V)$ recover
$Z^2_{\mathrm{CE}}(V_x, V_x)$, with the orbit and stack tangents corresponding precisely to
$B^2_{\mathrm{CE}}$ and $H^2_{\mathrm{CE}}$, respectively, just as in the classical deformation
theory~\cite{Fialowski1986,FialowskiPenkava2005a}.

We use the Ayupov--Omirov Killing form $\kappa_R$ defined on the Leibniz
variety~\cite{AyupovOmirov1998} and show that the open locus $U_{\Leib}$ where it is non-degenerate
is exactly the semisimple Lie stratum (Theorem~\ref{thm:Killing-Leibniz}). On this locus
$U_{\Lie} \subset \mathrm{Lie}(V)$, the classical Whitehead’s Lemma ($H^2_{\mathrm{CE}} = 0$)
implies rigidity. We thus deduce the known rigidity of semisimple Lie algebras
(e.g.~\cite{Khakimdjanov2000}) as a direct consequence of our geometric framework, showing
$U_{\Lie}$ is a finite union of open, isolated orbits in the quotient (Theorem~\ref{thm:rigid-ULie}
and Corollary~\ref{cor:det-open-rigid}).

\subsubsection*{Singular strata and intersection cohomology}

On the open loci $U_{\Alg}$, $U_{\Comm}$, and $U_{\Lie}$ we considered in this paper, the
associated moduli stacks $[\Alg(V)/\GL(V)]$, $[\Comm(V)/\GL(V)]$, and $[\Lie(V)/\GL(V)]$ are
formally smooth: at every point in these loci the relevant degree–$2$ cohomology groups ($HH^2$,
$\Harr^2$, and $H^2_{\mathrm{CE}}$, respectively) vanish, the $\GL(V)$-orbit is open in the
corresponding locus and rigid in moduli. Along these smooth strata, intersection cohomology does
not carry additional information because it coincides with ordinary cohomology.

In the Leibniz case, the Ayupov–Omirov Killing form $\kappa_R$ is nondegenerate on the open subset
$U_{\Leib}\subset\Leib(V)$. By Theorem~\ref{thm:Killing-Leibniz}, this locus consists exactly of
semisimple Lie brackets.  Then on $U_{\Leib}$ the Leibniz deformation theory reduces to the
Chevalley–Eilenberg theory and $H^2_{\Leib}$ vanishes just as in the Lie case. However, viewed
inside the full Leibniz variety $\Leib(V)$, the locus $U_{\Leib}$ lies in the proper closed
subvariety $\Lie(V)$ and is therefore far from generic: for a \emph{generic} Leibniz algebra the
Leibniz cohomology $H^2_{\Leib}$ does not vanish, and the corresponding strata in
$[\Leib(V)/\GL(V)]$ remain highly singular.

Outside the formally smooth loci $U_{\Alg}$, $U_{\Comm}$, $U_{\Lie}$, and $U_{\Leib}$, the relevant
degree–$2$ cohomology may no longer vanish, deformation theory ceases to be rigid, and singular
strata appear. In this singular regime, intersection cohomology is the natural receptacle for finer
geometric invariants. In particular, the incidence variety descriptions we construct here suggest a
way to realize the intersection complexes of the moduli spaces as diagonal slices of intersection
complexes on our incidence varieties. We will return to this point of view in future work.

\subsection*{What is known}

Over an algebraically closed field, Kaygorodov and Volkov gave a complete orbit classification and
degeneration graph for two dimensional algebras~\cite{KaygorodovVolkov2019}. The moduli problem in
dimensions dimensions three and four is studied by Fialowski and
Penkava~\cite{FialowskiPenkava2009b,FialowskiPenkava2015,FialowskiPenkavaPhillipson2011}, and in
dimension five by Happel~\cite{Happel1979} and Mazzola~\cite{Mazzola1979}.

These low-dimensional classification results stand in sharp contrast to the general case. Drozd’s
tame–wild dichotomy reveals that most classification problems are as hard as classifying pairs of
matrices up to simultaneous similarity~\cite{Drozd1980}. While Gabriel's representation-theoretic
approach provides a tractable path through representations of quivers~\cite{Gabriel1972}, the
inherent complexity of the moduli problem necessitates a shift from direct classification to
studying the geometric and deformation-theoretic properties of the varieties themselves.

The bridge between the moduli problem and deformation theory for associative algebras is the
Hochschild cohomology. Gerstenhaber identified Hochschild $2$-cocycles as classifying first-order
deformations and $3$-cocycles as carrying the primary obstructions~\cite{Gerstenhaber1964}. This
algebraic machinery was placed in a broader context by Nijenhuis and Richardson by showing that
deformations can be universally encoded as Maurer–Cartan elements in an appropriate differential
graded Lie algebra (DGLA)~\cite{NijenhuisRichardson1967}, a perspective that now underlies the
modern interpretation of formal moduli via $L_\infty$-algebras.

The crucial link from this algebraic approach to a geometric one was provided by Geiss and de la
Peña, who directly related Hochschild cohomology to the smoothness and singularities of
representation schemes and algebra varieties~\cite{GeissdelaPena1995}, and more recently by Chouhy
who reconciled Gerstenhaber's formal (analytic) deformations with algebro-geometric
degenerations~\cite{Chouhy2019}. This picture can be further refined by Schaps's analysis of
idempotents~\cite{Schaps1988}, Happel's study of derived invariants~\cite{Happel1987}, and
Gerstenhaber and Schack's criteria for rigidity in families using relative Hochschild
cohomology~\cite{GerstenhaberSchack1988}.

Within this associative framework, the commutative case must be handled separately with its
additional constraint. Harrison cohomology provides such a symmetric refinement, where $2$-classes
parameterize specifically commutative deformations and $3$-classes carry the
obstruction~\cite{Harrison1962} as expected. The precise relationship was explained by
Barr~\cite{Barr1968,Barr1969} and by Gerstenhaber and Schack~\cite{GerstenhaberSchack1988} by
constructing a Hodge-type decomposition of Hochschild cohomology and recovering Harrison cohomology
as the symmetric summand. An alternative but related interpretation of commutative deformation
theory was provided by André and Quillen via the cotangent complex and André–Quillen
cohomology~\cite{Andre1967,Quillen1970}.

The geometric manifestations of these deformations in moduli spaces are schemes or stacks. For
commutative algebras, Poonen constructed an affine scheme $B_n$ parameterizing rank-$n$ algebras
with a basis to analyze the number of components in its fibers~\cite{Poonen2008}, later extended by
O’Desky and Rosen to an equivariant setting~\cite{ODeskyRosen2023}.  Muro's work, set in a
homotopical context, conceptualizes these moduli spaces as stacks over operads, proves classical
results such as the fact that unital algebras form a Zariski open substack and explicitly
identifies the tangent complexes of these stacks with Hochschild-type cohomology
groups~\cite{Muro2014}.

This entire paradigm--from geometric classification to deformations governed by cohomology then
back to the geometry of moduli--is not unique to the associative case. One can develop a direct
parallel for Lie algebras where Chevalley–Eilenberg cohomology replaces Hochschild
cohomology. Nijenhuis and Richardson had already shown there is a DGLA whose 2- and 3-cocycles
control deformations and obstructions~\cite{NijenhuisRichardson1967}. This was used very
effectively by Fialowski, Fuchs, and Penkava to construct versal deformations and stratify the
moduli varieties of low-dimensional Lie
algebras~\cite{Fialowski1986,FialowskiFuchs1999,BodinFialowskiPenkava2005,FialowskiPenkava2007a,FialowskiPenkava2008a}. These
methods interface perfectly with the geometric study of varieties of Lie algebra laws, where
cohomological invariants are used to identify components and analyze
rigidity~\cite{GrunewaldOHalloran1988,Burde1999,Khakimdjanov2000}.

This framework extends to Leibniz algebras where we have a cohomology that controls the
infinitesimal deformations, just as its counterparts do~\cite{LodayPirashvili1993}. Consequently,
the geometric varieties of Leibniz laws have been analyzed in close analogy with the Lie setting,
using cohomological tools to understand their reducibility, components, and
degenerations~\cite{AncocheaBMS2007,Albeverio2006,AyupovOmirovRakhimov2020}.

\subsection*{Plan of the article}

In Section~\ref{sec:Associative}, we investigate the associative case. We construct the incidence
variety $\As(V)$, identify its diagonal with the classical variety $\Alg(V)$, and prove that its
fibers are isomorphic to spaces of Hochschild $2$-cocycles. We then analyze the separable locus
using the principal trace form and establish rigidity results for the moduli stack. We conclude the
section by refining these results for commutative algebras via Harrison cohomology.

In Section~\ref{sec:Leibniz}, we extend this framework to Leibniz and Lie algebras. We construct
the appropriate incidence variety for the Leibniz locus $\Leib(V)$ and identify its fibers with
Leibniz $2$-cocycles. We use the Killing form to characterize the semisimple locus and prove that
these algebras correspond to open orbits. Finally, we restrict these constructions to the Lie
subvariety, recovering the relationship between Chevalley--Eilenberg cohomology and the rigidity of
semisimple Lie algebras.

\section{The Associative Case}\label{sec:Associative}

\subsection{The quadratic associativity locus}

Let $V$ be a $\kk$-vector space of dimension $n$ with a fixed basis $(e_i)_{i=1}^n$.  The space
\[
  V^{2,1}:=\Hom(V\otimes V,V)\cong \kk^{n^3}
\]
parameterizes all bilinear products on $V$.  Each $x\in V^{2,1}$ determines a bilinear
multiplication $\mu_x\colon V\otimes V\to V$ by
\[
  \mu_x(e_i,e_j)=\sum_{\ell=1}^n x_{i,j}^{\ell}e_\ell,
\]
and conversely any bilinear product arises uniquely in this way.\footnote{Please note that the
  index $\ell$ in $x_{i,j}^\ell$ is just another index, and does not indicate an exponent.}

\begin{proposition}\label{prop:assoc-quadrics}
  Let $V^{3,1}:=\Hom(V^{\otimes 3},V)$ and define $\Phi\colon V^{2,1}\to V^{3,1}$ by
  \[
    \Phi(x)(a,b,c):=\mu_x(\mu_x(a,b),c)-\mu_x(a,\mu_x(b,c)).
  \]
  The associativity locus
  \[
    \Alg(V):=\{x\in V^{2,1}:\Phi(x)=0\}
  \]
  is a closed subscheme cut out by $n^4$ quadratic polynomials
  \[
    F^m_{i,j,k}(x)=\sum_{\ell=1}^n x_{i,j}^{\ell}x_{\ell,k}^m-\sum_{\ell=1}^n x_{i,\ell}^m x_{j,k}^{\ell}.
  \]
\end{proposition}


\subsection{Bilinearization and its symmetries}\label{sec:bilinearization}

Define a bilinear map $\beta\colon V^{2,1}\times V^{2,1}\to V^{3,1}$ by
\begin{equation}\label{eq:beta}
  \beta(x,y)(a,b,c):=\mu_x(\mu_y(a,b),c)-\mu_x(a,\mu_y(b,c)).
\end{equation}
In coordinates this reads
\[
  \beta^m_{i,j,k}(x,y)
  =\sum_{\ell} x^{\ell}_{i,j} y^m_{\ell,k}-\sum_{\ell} y^m_{i,\ell}x^{\ell}_{j,k}.
\]
By definition $\Phi(x)=\beta(x,x)$.  

For each $\varphi\in (V^{3,1})^\vee$ there exists a matrix $q_\varphi\in \Mat_{n^3}(\kk)$ such that
\[
  \langle \varphi,\beta(x,y)\rangle  =  x^{t} q_\varphi y
\]
where $x,y$ are viewed as column vectors of length $n^3$ collecting the structure constants.  Let
$Q_{\As}$ denote the linear span of $\{q_\varphi:\varphi\in (V^{3,1})^\vee\}$ in $\Mat_{n^3}(\kk)$.
Since the characteristic of $\kk$ is $0$, the skew-symmetric part of $q_\varphi$ does not
contribute to the diagonal values $x^{t}q_\varphi x$, so we can replace $Q_{\As}$ by its symmetric
part
\[
  Q_{\As}^{\rm sym}=\Span\{\tfrac12(q+q^{t}):q\in Q_{\As}\}
\]
without changing the equations on the diagonal.  We therefore define the bilinear incidence
subvariety
\[
  \As(V) := \{(x,y)\in V^{2,1}\times V^{2,1}: x^{t}qy=0 \text{ for all } q\in Q_{\As}^{\rm sym}\}.
\]

\begin{proposition}\label{prop:diag-equals-assoc}
  The diagonal $\Delta=\{(x,x):x\in V^{2,1}\}$ satisfies
  \[
    \Alg(V) = \{x\in V^{2,1}: (x,x)\in \As(V)\}.
  \]
\end{proposition}

\begin{proof}
  For each $\varphi\in (V^{3,1})^\vee$ we have
  \[
    F_\varphi(x)
    := \langle \varphi,\Phi(x) \rangle
    = \langle \varphi,\beta(x,x) \rangle
    = x^{t}q_\varphi x.
  \]
  The ideal of $\Alg(V)$ is generated by the polynomials $F_\varphi$ as $\varphi$ ranges over a
  basis of $(V^{3,1})^\vee$.  Thus $\Phi(x)=0$ if and only if $x^{t}q_\varphi x=0$ for all
  $\varphi$, which is equivalent to $(x,x)\in \As(V)$.  Since skew-symmetric matrices vanish on the
  diagonal, this condition depends only on the symmetric part of $Q_{\As}$.
\end{proof}

\begin{proposition}\label{prop:Q-rep}
  The space $Q_{\As}^{\rm sym}$ is a $\GL(V)$-subrepresentation of $\Sym^2((V^{2,1})^\vee)$ under
  the natural $\GL(V)$-action on $V^{2,1}$ and $V^{3,1}$.
\end{proposition}

\begin{proof}
  The group $\GL(V)$ acts on $V^{2,1} = \Hom(V^{\otimes 2},V)$ and on $V^{3,1} = \Hom(V^{\otimes 3},V)$
  by transport of structure.  Concretely, for $g\in \GL(V)$ and $\mu\in V^{2,1}$ one has
  \[
    (g\cdot\mu)(a,b) = g \mu(g^{-1}a,g^{-1}b),
  \]
  and an analogous formula holds in arity $3$.  A direct computation shows that $\beta$ is
  $\GL(V)$-equivariant:
  \[
    \beta(g\cdot x,g\cdot y) = g\cdot\beta(x,y)
  \]
  for all $g,x,y$.  Dualizing, we obtain
  \[
    \langle g\cdot \varphi,\beta(g\cdot x,g\cdot y) \rangle
    = \langle \varphi,\beta(x,y) \rangle
    =x^{t}q_\varphi y.
  \]
  On the other hand the left-hand side is equal to
  \[
    \langle g\cdot \varphi, \beta(g\cdot x,g\cdot y) \rangle
    = (g\cdot x)^{t} q_{g\cdot \varphi} (g\cdot y).
  \]
  Writing the action on $V^{2,1}$ in coordinates shows that this implies a congruence relation
  \[
    q_{g\cdot\varphi} = (g^{-1})^{t} q_\varphi g^{-1}.
  \]
  Thus the span of the matrices $q_\varphi$ is stable under the induced action of $\GL(V)$ inside
  $\End(V^{2,1})$, and this action preserves symmetry.  Therefore $Q_{\As}^{\rm sym}$ is a
  $\GL(V)$-subrepresentation of $\Sym^2((V^{2,1})^\vee)$.
\end{proof}

\subsection{The incidence as a sheaf and its cohomological fiber}

Let $X : =V^{2,1}$ with coordinates $x$.  Choose a basis $\{q_1,\dots,q_M\}$ of
$Q_{\As}^{\rm sym}$.  Define a morphism of vector bundles
\[
  \Psi\colon  \cO_X^{\oplus n^3}\longrightarrow \cO_X^{\oplus M}
\]
by sending a section $y$ to the tuple
\[
  \Psi(y) = (x^{t} q_1 y,\dots,x^{t} q_{M} y),
\]
where each entry is a linear polynomial in the coordinates of $x$ and in the fiber coordinates $y$.
Set $\mathcal{K}:=\ker(\Psi)$.  For $x\in X$ the fiber of $\mathcal{K}$ at $x$ is
\[
  \mathcal{K}(x) = \{y\in V^{2,1}: x^{t} qy=0 \text{ for all } q\in Q_{\As}^{\rm sym}\}=:F_x,
\]
and the total space of $\mathcal{K}$ identifies with the incidence variety $\As(V)$.

We next relate $F_x$ to the Zariski tangent space of $\Alg(V)$ and to Hochschild cohomology.

\begin{proposition}\label{prop:fiber-Z2}
  Let $x\in \Alg(V)$ and let $V_x:=(V,\mu_x)$.  Then the Zariski tangent space of $\Alg(V)$ at $x$
  coincides with $F_x$ and with the space of Hochschild $2$-cocycles:
  \[
    T_x\Alg(V) = F_x = Z^2(V_x,V_x).
  \]
\end{proposition}

\begin{proof}
  First we identify $F_x$ with the Zariski tangent space.  The scheme $\Alg(V)$ is defined by the
  vanishing of the quadratic polynomials
  \[
    F_\varphi(x) : =\langle \varphi,\Phi(x) \rangle = x^{t}q_\varphi x
  \]
  as $\varphi$ ranges over a basis of $V(^{3,1})^\vee$.  For $y\in V^{2,1}$ the directional
  derivative of $F_\varphi$ at $x$ in the direction $y$ is
  \[
    \mathrm{d}F_\varphi(x)[y]
    = \left.\frac{\mathrm{d}}{\mathrm{d}t}\right|_{t=0} (x+ty)^{t}q_\varphi(x+ty)
    = x^{t}q_\varphi y + y^{t}q_\varphi x
    = 2 x^{t}q_\varphi y,
  \]
  because $q_\varphi$ is symmetric and the characteristic $\kk$ is $0$.  Hence
  \[
    T_x\Alg(V)
    = \bigcap_{\varphi} \ker \left(\mathrm{d} F_\varphi(x) \right)
    = \{y\in V^{2,1} : x^{t}q_\varphi y=0 \text{ for all } \varphi \}
    =F_x.
  \]
  
  We now identify $F_x$ with the space of Hochschild $2$-cocycles.  Let $C^\ast(V_x,V_x)$ denote
  the Hochschild cochain complex, and let $d_{\mu_x}\colon C^2\to C^3$ be the Hochschild
  differential.  By definition of the Gerstenhaber bracket, the differential satisfies
  \[
    d_{\mu_x}(y)=[\mu_x,y]
  \]
  for any $2$-cochain $y\in C^2(V_x,V_x)$.  A direct expansion of the insertion operations in the
  bracket shows that
  \[
    [\mu_x,y](a,b,c)
    = \mu_x(y(a,b),c) - \mu_x(a,y(b,c))
    + y(\mu_x(a,b),c) - y(a,\mu_x(b,c))
  \]
  for all $a,b,c\in V$, which we may write as
  \[
    [\mu_x,y] = \beta(x,y) + \beta(y,x).
  \]
  Thus $d_{\mu_x}(y) = 0$ if and only if $\beta(x,y) + \beta(y,x) = 0$.  For each
  $\varphi\in (V^{3,1})^\vee$ we have
  \[
    \langle \varphi, \beta(x,y) + \beta(y,x) \rangle
    = \langle \varphi, \beta(x,y) \rangle + \langle \varphi, \beta(y,x) \rangle
    = x^{t}q_\varphi y + y^{t}q_\varphi x
    = 2 x^{t}q_\varphi y.
  \]
  Therefore $d_{\mu_x}(y)=0$ if and only if $x^{t}q_\varphi y=0$ for all $\varphi$, which is
  equivalent to $y\in F_x$.  Hence
  \[
    F_x = \ker d_{\mu_x} = Z^2(V_x,V_x),
  \]
  as claimed.
\end{proof}

\subsection{The principal trace form and semisimplicity}

For $x\in \Alg(V)$, define the principal trace form $T_x\colon V\times V\to \kk$ by
\[
  T_x(a,b) := \Tr(L_{ab}),
\]
where $L_z\colon V_x\to V_x$ is left multiplication by $z$.  Let $D_x = (\Tr(L_{e_ie_j}))_{i,j}$ be
the Gram matrix of $T_x$ in the basis $(e_i)$ and set $\Delta(x) := \det D_x$.

\begin{theorem}[{\cite[Thm.~3.1 and Cor.~3.1]{Aguiar1998}}]\label{thm:Aguiar}
  For a finite-dimensional algebra $A$ over a field of characteristic $0$, the principal trace form
  $T_A$ is nondegenerate if and only if $A$ is semisimple.
\end{theorem}

Over a field of characteristic $0$, a finite-dimensional algebra is semisimple if and only if it is separable.  The open subset
\[
  U := \{ x\in \Alg(V): \Delta(x)\neq 0 \}
\]
therefore consists precisely of the separable algebras.  For any irreducible component $Z$ of
$\Alg(V)$ that meets $U$, the intersection $U\cap Z$ is a nonempty Zariski-open subset of $Z$ and
hence is dense in $Z$.

\subsection{Tangent spaces, rigidity, and moduli}

We now collect the main geometric consequences of the constructions above.

\begin{theorem}\label{thm:tangent-generic}
  Let $U\subset \Alg(V)$ be the open separable locus.  For $x\in U$, the Zariski tangent space satisfies
  \[
    T_x\Alg(V) \cong Z^2(V_x,V_x)
  \]
  and
  \[
    \dim_\kk T_x\Alg(V)  =  n^2-\dim_\kk\Der_\kk(V_x)  =  n^2-n+\dim_\kk Z(V_x).
  \]
  When $\kk$ is algebraically closed and
  \[
    V_x\cong \bigoplus_{s=1}^k M_{r_s}(\kk)
    \qquad\text{with}\qquad
    \sum_s r_s^2=n,
  \]
  this dimension is equal to $n^2 - n + k$.
\end{theorem}

\begin{proof}
  Proposition~\ref{prop:fiber-Z2} gives $T_x\Alg(V)\cong Z^2(V_x,V_x)$.  Since $V_x$ is separable,
  $V_x$ is projective as a module over its enveloping algebra
  $V_x^e := V_x\otimes_\kk V_x^{\mathrm{op}}$.  Hence by \cite[Prop.~9.1.10]{Weibel1994}
  \[
    HH^m(V_x,V_x)\cong \Ext_{V_x^e}^m(V_x,V_x) = 0
  \]
  for all $m>0$.  In particular $HH^1(V_x,V_x) = 0$ and $HH^2(V_x,V_x) = 0$.

  The Hochschild complex in degrees $1$ and $2$ is
  \[
    C^1(V_x,V_x) = \Hom_\kk(V_x,V_x),
    \qquad
    C^2(V_x,V_x) = \Hom_\kk(V_x^{\otimes 2},V_x),
  \]
  with differential $d^1\colon C^1\to C^2$ given by
  \[
    (d^1 f)(a,b) = a f(b) - f(a b) + f(a) b.
  \]
  It is standard that $\ker d^1$ is the space of derivations $\Der_\kk(V_x)$ and that
  $B^2(V_x, V_x) = \mathrm{im}(d^1)$.  Since $HH^2(V_x,V_x) = 0$ and $HH^2 \cong Z^2/B^2$, one has
  $Z^2(V_x, V_x) = B^2(V_x, V_x)$ and therefore
  \[
    \dim_\kk Z^2(V_x, V_x)
    = \dim_\kk B^2(V_x, V_x)
    = \dim_\kk C^1(V_x, V_x) - \dim_\kk\ker d^1
    = n^2 - \dim_\kk\Der_\kk(V_x).
  \]
  
  For a separable algebra $V_x$ one also has $HH^1(V_x,V_x)=0$, so every derivation is inner.  The map
  \[
    V_x\longrightarrow \Der_\kk(V_x),\quad
    a\longmapsto [a,-],
  \]
  has kernel equals to the center $Z(V_x)$, and inner derivations are isomorphic as a vector space
  to $V_x/Z(V_x)$.  Thus
  \[
    \dim_\kk\Der_\kk(V_x) = \dim_\kk V_x - \dim_\kk Z(V_x) = n - \dim_\kk Z(V_x),
  \]
  and the claimed formula for $\dim_\kk T_x\Alg(V)$ follows.  Over an algebraically closed field, a
  separable algebra decomposes as $\bigoplus_s M_{r_s}(\kk)$, and the center has dimension $k$.
  Hence $\dim_\kk T_x\Alg(V) = n^2-n+k$ in that case.
\end{proof}

\begin{theorem}\label{thm:rigidity-moduli}
  Assume $\kk$ is algebraically closed.  For $x\in U$, the moduli stack tangent vanishes:
  \[
    T_{[V_x]}[\Alg(V)/\GL(V)] \cong HH^2(V_x, V_x) = 0.
  \]
  Consequently, each isomorphism class $[V_x]$ is an isolated point of the coarse moduli space
  $\Alg(V)\sslash \GL(V)$ in any neighborhood that meets the separable locus.
\end{theorem}

\begin{proof}
  Gerstenhaber showed that infinitesimal deformations of an associative algebra modulo isomorphism
  are governed by the Hochschild cohomology group $HH^2(V_x, V_x)$ \cite{Gerstenhaber1964}.  For a
  separable algebra this group vanishes.  Thus the tangent space to the deformation functor, and
  hence to the moduli stack $[\Alg(V)/\GL(V)]$ at the point $[V_x]$, is zero.

  Luna’s étale slice theorem \cite[§III.1]{Luna1973} describes a neighborhood of a closed
  $\GL(V)$-orbit in $\Alg(V)$ as an étale quotient of a normal slice by the stabilizer.  When the
  stack tangent vanishes, the slice has trivial tangent space at the base point, so the quotient
  has no nontrivial tangent directions.  Therefore $[V_x]$ is an isolated point of the coarse
  moduli space in any neighborhood that meets the separable locus.
\end{proof}

\begin{theorem}\label{thm:diagonal-stratification}
  The loci where $\rk d_{\mu_x}$ is constant yield a finite determinantal stratification of
  $\Alg(V)$ into locally closed $\GL(V)$-stable subvarieties.  The orbit stratification by
  isomorphism class refines this further.  Over an algebraically closed field, the separable locus
  stratifies by block-size profiles
  \[
    V_x \cong \bigoplus_s M_{r_s}(\kk)
    \qquad\text{with}\qquad
    \sum_s r_s^2=n,
  \]
  while over a general field of characteristic $0$ the stratification is by Wedderburn types
  $V_x\cong \bigoplus_i M_{r_i}(D_i)$ with division algebras $D_i$.  The open top stratum consists
  of separable algebras characterized by $\Delta(x)\neq 0$.
\end{theorem}

\begin{proof}
  For each $x\in \Alg(V)$, the Hochschild differential
  $d_{\mu_x}\colon C^2(V_x,V_x)\to C^3(V_x,V_x)$ depends polynomially on the structure constants of
  $\mu_x$.  Choosing bases for $C^2$ and $C^3$, the matrix entries of $d_{\mu_x}$ become polynomial
  functions of $x$, and the rank function $x\mapsto \rk d_{\mu_x}$ is upper semicontinuous.  The
  loci where this rank is bounded above by a fixed integer are closed determinantal conditions, and
  the differences of consecutive such loci produce a finite stratification by locally closed sets
  on which the rank is constant.  The $\GL(V)$-equivariance of $d_{\mu_x}$ implies that these
  strata are $\GL(V)$-stable.

  Two points $x,x'\in \Alg(V)$ in the same $\GL(V)$-orbit define isomorphic algebras $V_x$ and
  $V_{x'}$ and therefore have isomorphic Hochschild complexes with differentials of the same rank.
  Thus each isomorphism class lies inside a single rank stratum, and the orbit stratification
  refines the determinantal stratification.  The description of strata on the separable locus
  follows from the Wedderburn structure theorem: over an algebraically closed field, $V_x$
  decomposes as $\bigoplus_s M_{r_s}(\kk)$ and the tuple $(r_s)$ determines the algebra up to
  isomorphism and hence the orbit type, while over a general field of characteristic $0$ one
  obtains Wedderburn components $M_{r_i}(D_i)$.  Theorem~\ref{thm:Aguiar} identifies the subset
  where $\Delta(x)\neq 0$ with the semisimple, hence separable, algebras, and this locus forms the
  unique open top stratum on each irreducible component that it meets.
\end{proof}

By Theorem~\ref{thm:tangent-generic}, tangent spaces along the separable stratum have dimension
$n^2 - n + \dim Z(V_x)$.  Theorem~\ref{thm:rigidity-moduli} shows that these directions are
precisely $\GL(V)$-orbit directions and vanish upon passage to isomorphism classes.

\begin{corollary}\label{cor:count-assoc-orbits}
  Assume $\kk$ is algebraically closed of characteristic $0$.  The number of open $\GL(V)$-orbits
  in the separable locus $U\subset \Alg(V)$ is equal to the number of partitions of $n$ as an
  ordered sum of squares:
  \[
    N_{\mathrm{assoc}}(n) := \#\left\{(r_1,\ldots,r_k) : r_1 \ge \cdots \ge r_k \ge 1,\ \sum_{s=1}^k r_s^2 = n\right\}.
  \]
\end{corollary}

\begin{remark}
  The numbers $N_{\mathrm{assoc}}(n)$ for small $n$ form the sequence \cite[A001156]{OEIS}
  \[
    N_{\mathrm{assoc}}(1),\ldots,N_{\mathrm{assoc}}(20): 1, 1, 1, 1, 2, 2, 2, 2, 3, 4, 4, 4, 5, 6, 6, 6, 8, 9, 10, 10.
  \]
\end{remark}

\subsection{The commutative case, Harrison cohomology, and rigidity}\label{sec:Comm-quotient-rigidity}

Impose commutativity via the linear ideal
\[
  I_{\mathrm{comm}} := \langle x^\ell_{i,j}-x^\ell_{j,i}\rangle
  \subset \kk[x_{i,j}^\ell]
\]
and define the commutative associative locus
\[
  \Comm(V) := \mathrm{V}(I_{\mathrm{comm}}) \cap \Alg(V).
\]
All constructions from \S\ref{sec:bilinearization} restrict naturally to this linear subspace of
$V^{2,1}$.  For $x\in \Comm(V)$, the Hochschild description of the tangent space becomes
\[
  T_x\Comm(V)
  = \{ \varphi\in Z^2(V_x, V_x): \varphi(a,b) = \varphi(b,a) \}.
\]
In other words, $T_x\Comm(V)$ consists of those Hochschild $2$-cocycles that are symmetric in the
two arguments.

For a commutative algebra $A$, Harrison introduced a refinement of Hochschild cohomology whose
$2$-cocycles classified infinitesimal commutative deformations~\cite{Harrison1962}.  The Harrison
complex can be described as a quotient of the Hochschild complex by shuffle relations, and in
degree $2$ the natural map from symmetric Hochschild $2$-cochains to the Harrison complex is an
isomorphism~\cite{Barr1968,Barr1969,GerstenhaberSchack1988}.  Thus, for $x\in \Comm(V)$, the
symmetric $2$-cocycles above identify canonically with Harrison $2$-cocycles:
\[
  T_x\Comm(V) \cong  \Harr^2(V_x,V_x),
\]
in agreement with the general deformation theory of commutative algebras and its André–Quillen
reinterpretation~\cite{Andre1967,Quillen1970}.

The principal trace form from the previous subsection restricts to $\Comm(V)$ with the same
discriminant $\Delta(x) = \det D_x$.  By Aguiar's theorem, the open subset
\[
  U_{\Comm} := \{x\in \Comm(V): \Delta(x)\neq 0\}
\]
consists precisely of separable commutative algebras.  Over a field of characteristic $0$ these are
exactly the finite étale $\kk$-algebras of dimension $n$, that is, finite products of finite
separable field extensions.

\begin{theorem}\label{thm:rigid-UComm}
  For $x\in U_{\Comm}$, the algebra $V_x$ is finite étale and satisfies $HH^m(V_x,V_x)=0$ and
  $\Harr^m(V_x,V_x) = 0$ for $m>0$.  Consequently,
  \[
    T_x\Comm(V) = T_x(\GL(V) \cdot x)
    \quad\text{and}\quad
    T_{[x]}[\Comm(V)/\GL(V)] = 0.
  \]
  Thus each $\GL(V)$-orbit in $U_{\Comm}$ is Zariski open, and $[x]$ is isolated in the coarse
  moduli space $\Comm(V)\sslash \GL(V)$.
\end{theorem}

\begin{proof}
  If $x\in U_{\Comm}$ then $V_x$ is separable and commutative, hence finite étale of dimension
  $n$ over $\kk$.  As in the proof of Theorem~\ref{thm:tangent-generic}, separability implies $V_x$
  is projective over $V_x^e$, and therefore $HH^m(V_x,V_x) = 0$ for $m>0$.  The André–Quillen
  interpretation of commutative deformation theory \cite{Andre1967,Quillen1970} shows that the
  cotangent complex of a finite étale algebra is concentrated in degree $0$, so all higher
  André–Quillen and Harrison cohomology groups vanish.  In particular, $\Harr^2(V_x,V_x) = 0$.

  The identification $T_x\Comm(V)\cong \Harr^2(V_x, V_x)$ then shows $T_x\Comm(V)=0$ modulo the
  tangent directions coming from the $\GL(V)$-orbit.  More concretely, since $HH^2(V_x, V_x)=0$,
  every Hochschild $2$-cocycle is a coboundary.  Intersecting with the symmetry condition shows
  that every element of $T_x\Comm(V)$ arises from the derivative of an inner automorphism, so
  \[
    T_x\Comm(V) = T_x(\GL(V) \cdot x).
  \]
  Hence the orbit accounts for all tangent directions in $\Comm(V)$ at $x$, which means that the
  orbit is open in $\Comm(V)$.  Passing to the quotient stack and then to the coarse moduli space
  removes these orbit directions, so the tangent space at $[x]$ vanishes and $[x]$ is isolated.
\end{proof}

\begin{corollary}\label{cor:UComm-orbits}
  The locus $U_{\Comm}$ is a finite union of open $\GL(V)$-orbits indexed by isomorphism classes of
  finite étale $\kk$-algebras of dimension $n$.  Over an algebraically closed field, $U_{\Comm}$
  is a single open orbit corresponding to $\kk^n$ with coordinate-wise product.
\end{corollary}

\begin{proof}
  Finite étale $\kk$-algebras of dimension $n$ are classified up to isomorphism, and two points
  $x,x'\in U_{\Comm}$ lie in the same $\GL(V)$-orbit if and only if the corresponding algebras
  $V_x$ and $V_{x'}$ are isomorphic as $\kk$-algebras.  Thus the $\GL(V)$-orbits in $U_{\Comm}$ are
  in bijection with isomorphism classes of finite étale $\kk$-algebras of dimension $n$.  When
  $\kk$ is algebraically closed, every finite étale algebra of dimension $n$ is isomorphic to
  $\kk^n$, so there is a single such orbit.
\end{proof}

Since finite étale algebras have no nonzero derivations, the formula of
Theorem~\ref{thm:tangent-generic} shows that $\dim T_x\Alg(V) = n^2$ at points of $U_{\Comm}$.  The
same calculation and the symmetry constraint give $\dim T_x\Comm(V) = n^2$ there.  Thus $\Comm(V)$
is smooth of dimension $n^2$ along $U_{\Comm}$ and each $\GL(V)$-orbit in $U_{\Comm}$ is open.

\section{The Leibniz Case}\label{sec:Leibniz}

\subsection{The quadratic Leibniz locus}\label{sec:Leibniz-locus}

We work with the right Leibniz identity
\[
  \mu(\mu(a,b),c) = \mu(\mu(a,c),b) + \mu(a,\mu(b,c)).
\]
As in the associative case, we view $V^{2,1} := \Hom(V\otimes V,V)$ and
$V^{3,1} := \Hom(V^{\otimes 3},V)$ as affine spaces of structure constants.  For $x\in V^{2,1}$ let
$\mu_x$ be the corresponding bilinear law and define
\[
  F\colon V^{2,1}\longrightarrow V^{3,1},\qquad
  F(x)(a,b,c) := \mu_x(\mu_x(a,b),c) - \mu_x(\mu_x(a,c),b) - \mu_x(a,\mu_x(b,c)).
\]

\begin{proposition}\label{prop:Leibniz-quadrics}
  The Leibniz locus
  \[
    \Leib(V) : =\{ x\in V^{2,1} : F(x) = 0 \}
  \]
  is a closed subscheme cut out by $n^4$ quadratic polynomials
  \[
    G^m_{i,j,k}(x)
    = \sum_{\ell=1}^n  x^\ell_{i,j}  x^m_{\ell,k}
    - \sum_{\ell=1}^n  x ^\ell_{i,k} x^m_{\ell,j}
    - \sum_{\ell=1}^n  x^m_{i,\ell} x^\ell_{j,k}.
  \]
\end{proposition}


\subsection{Bilinearization and cohomological interpretation}\label{sec:Leib-bilinearization}

Define a bilinear map $B \colon V^{2,1}\times V^{2,1}\to V^{3,1}$ by
\[
  B(x,y)(a,b,c) := \mu_x(\mu_y(a,b),c) - \mu_x(\mu_y(a,c),b) - \mu_y(a,\mu_x(b,c)),
\]
where for a $2$-cochain $y$ we write $\mu_y(a,b):=y(a,b)$.  In structure constants this is
\[
  B^m_{i,j,k}(x,y)
  = \sum_{\ell} \left(x^\ell_{i,j} y^m_{\ell,k} - x^\ell_{i,k} y^m_{\ell,j} - x^m_{i,\ell} y^\ell_{j,k}\right).
\]
By definition $F(x)=B(x,x)$.

\begin{proposition}\label{prop:Leib-diagonal}
  For $x\in \Leib(V)$, the Zariski tangent space is
  \[
    T_x\Leib(V) = \{ y\in V^{2,1}: B(x,y) + B(y,x) = 0 \}.
  \]
\end{proposition}

\begin{proof}
  For $x,y\in V^{2,1}$ we have
  \[
    F(x+ty) = B(x+ty, x+ty)
    = B(x,x) + t \bigl( B(x,y) + B(y,x) \bigr) + t^2 B(y,y),
  \]
  by bilinearity of $B$.  If $x\in\Leib(V)$ then $B(x,x) = 0$, so
  \[
    F(x+ty) = t \bigl( B(x,y) + B(y,x) \bigr) + O(t^2).
  \]
  Thus the directional derivative of $F$ at $x$ in direction $y$ is
  \[
    \mathrm dF_x(y) = B(x,y) + B(y,x),
  \]
  and $y$ is tangent to $\Leib(V)$ at $x$ if and only if this derivative vanishes.  This yields the
  stated description of $T_x\Leib(V)$.
\end{proof}

For a Leibniz algebra $(V,\mu_x)$, let $C^p(V,V) := \Hom(V^{\otimes p},V)$ be the right Leibniz
cochain space with coefficients in the adjoint bimodule.  The degree-$2$ coboundary
\[
  \delta_x\colon C^2(V,V)\longrightarrow C^3(V,V)
\]
is given by~\cite{Loday1993,LodayPirashvili1993}
\begin{align*}
  (\delta_x\varphi)(a,b,c)
  = & \mu_x(\varphi(a,b),c) - \mu_x(\varphi(a,c),b) - \mu_x(a,\varphi(b,c))\\
    & + \varphi(\mu_x(a,b),c) - \varphi(\mu_x(a,c),b) - \varphi(a,\mu_x(b,c)).
\end{align*}

\begin{lemma}\label{lem:Leib-delta-B}
  For $x\in \Leib(V)$ and $y\in C^2(V,V)$ we have
  \[
    \delta_x y  =  B(x,y)+B(y,x).
  \]
\end{lemma}

\begin{proof}
  Using the notation $\mu_x$ for the Leibniz product and writing $y(a,b)$ for a general
  $2$-cochain, we compute
  \begin{align*}
    B(x,y)(a,b,c)
    & = \mu_x(y(a,b),c) - \mu_x(y(a,c),b) - y(a,\mu_x(b,c)),\\
    B(y,x)(a,b,c)
    & = y(\mu_x(a,b),c) - y(\mu_x(a,c),b) - \mu_x(a,y(b,c)).
  \end{align*}
  Adding these gives
  \begin{align*}
    B(x,y)(a,b,c)+B(y,x)(a,b,c)
    = & \mu_x(y(a,b),c) - \mu_x(y(a,c),b) - \mu_x(a,y(b,c))\\
      & + y(\mu_x(a,b),c) - y(\mu_x(a,c),b) - y(a,\mu_x(b,c)),
  \end{align*}
  which coincides with $(\delta_x y)(a,b,c)$ by the defining formula above.  Hence
  $\delta_x y = B(x,y)+B(y,x)$.
\end{proof}

\begin{proposition}\label{prop:Leib-tangent=Z2}
  For $x\in \Leib(V)$, the Zariski tangent space is naturally isomorphic to the space of Leibniz
  $2$-cocycles
  \[
    T_x\Leib(V) \cong Z^2_{\mathrm{Leib}}(V,V):=\ker(\delta_x).
  \]
  Under the $\GL(V)$-action on $V^{2,1}$, the orbit tangent at $x$ identifies with the
  $2$-coboundaries
  \[
    T_x(\GL(V) \cdot x)
    \cong B^2_{\mathrm{Leib}}(V,V)
    := \mathrm{im}(\delta_x\colon C^1\to C^2),
  \]
  and the tangent of the quotient stack is Leibniz cohomology
  \[
    T_{[x]}[\Leib(V)/\GL(V)]
    \cong H^2_{\mathrm{Leib}}(V,V)
    := Z^2_{\mathrm{Leib}}(V,V)/B^2_{\mathrm{Leib}}(V,V).
  \]
\end{proposition}

\begin{proof}
  By Proposition~\ref{prop:Leib-diagonal} and Lemma~\ref{lem:Leib-delta-B}, we have
  \[
    T_x\Leib(V)
    = \{y\in C^2(V,V) : \delta_x y = 0 \}
    = Z^2_{\mathrm{Leib}}(V,V).
  \]
  
  To identify the orbit tangent, consider the action of $G := \GL(V)$ on $V^{2,1}$ by transport of
  structure:
  \[
    (g\cdot\mu)(a,b) := g \mu(g^{-1}a,g^{-1}b).
  \]
  The Lie algebra $\mathfrak{gl}(V)=\End_\kk(V)$ acts by derivation of this formula.  Let
  $f\in C^1(V,V) = \End_\kk(V)$ and consider a first-order deformation $g_t = \id + t f + O(t^2)$.
  Then
  \[
    (g_t\cdot\mu_x)(a,b)
    = \mu_x(a,b) + t \bigl( f(\mu_x(a,b)) - \mu_x(f(a),b) - \mu_x(a,f(b)) \bigr) + O(t^2).
  \]
  The linear term is $-(\delta_x f)(a,b)$, where $\delta_x\colon C^1\to C^2$ is the Leibniz
  coboundary in degree $1$.  Thus the tangent directions coming from the $G$-orbit are exactly the
  $2$-coboundaries:
  \[
    T_x(\GL(V) \cdot x) = \mathrm{im}(\delta_x\colon C^1\to C^2)=B^2_{\mathrm{Leib}}(V,V).
  \]

  Finally, the tangent space to the quotient stack $[\Leib(V)/\GL(V)]$ at $[x]$ is, by general
  deformation theory, the tangent space of deformations modulo those induced by the group action.
  This is precisely the quotient
  \[
    Z^2_{\mathrm{Leib}}(V,V)/B^2_{\mathrm{Leib}}(V,V) = H^2_{\mathrm{Leib}}(V,V).
  \]
  For background on Leibniz cohomology and its deformation-theoretic interpretation, see
  Loday--Pirashvili \cite{LodayPirashvili1993} and Loday \cite{Loday1993}.
\end{proof}

\subsection{Canonical cocycles and the Killing form}\label{sec:Leib-cocycles}

For a fixed $x\in\Leib(V)$, write $L_a(b) := \mu_x(a,b)$ and $R_a(b) := \mu_x(b,a)$ for left and
right multiplication operators on $V$.  The right Leibniz identity is equivalent to the statement
that each $R_c$ is a derivation of $\mu_x$ in the first variable:
\[
  R_c(\mu_x(a,b)) = \mu_x(R_c(a),b) + \mu_x(a,R_c(b)),
\]
that is,
\[
  \mu_x(\mu_x(a,b),c) = \mu_x(\mu_x(a,c),b) + \mu_x(a,\mu_x(b,c)).
\]

The {\em Leibniz kernel} $\Leib(V_x)$ is the subspace spanned by the squares
$\mu_x(a,a)$~\cite{Loday1993,Cuvier1994}.  In the right Leibniz case one directly checks that this
subspace lies in the right annihilator.  Indeed, for any $a,x\in V$ we have
\[
  \mu_x(\mu_x(x,a),a) = \mu_x(\mu_x(x,a),a) + \mu_x(x,\mu_x(a,a))
\]
by the Leibniz identity with $(a,b,c)=(x,a,a)$, which forces $\mu_x(x,\mu_x(a,a)) = 0$.  Hence
\[
  R_{\mu_x(a,a)} = 0  \qquad\text{for all }a\in V,
\]
and every element of the Leibniz kernel has zero right multiplication.

Define the {\em right modular character}
\[
  \sigma_R\colon V\to \kk,\qquad \sigma_R(a) := \Tr(R_a),
\]
and similarly the left modular character $\sigma_L(a) := \Tr(L_a)$.  For the trivial
one-dimensional bimodule $\kk$ (with zero left and right actions), the Leibniz coboundary in degree
$1$ takes the form
\[
  (\delta\sigma_R)(a,b) = -\sigma_R(\mu_x(a,b)).
\]
Using the derivation property of $R_c$ and the right Leibniz identity, one checks
\[
  R_{\mu_x(a,b)} = -[R_a,R_b],
\]
so that
\[
  (\delta\sigma_R)(a,b) = -\Tr(R_{\mu_x(a,b)}) = \Tr([R_a,R_b]) = 0.
\]
Hence $\sigma_R$ is a Leibniz $1$-cocycle.  An analogous argument using the relation
\[
  [R_b,L_a] = L_{\mu_x(a,b)}
\]
shows that $\sigma_L$ is a $1$-cocycle as well.  These give $\GL(V)$-equivariant maps
$\Leib(V)\to V^\ast$ whose zero loci cut out the right and left unimodular subvarieties.

We now consider the Killing form defined by Ayupov and Omirov in~\cite{AyupovOmirov1998} for
Leibniz algebras that governs the semisimple Lie stratum. They defined the Killing form as a
symmetric bilinear form
\[
  \kappa_R\colon V\times V\to \kk,\qquad \kappa_R(a,b):=\Tr(R_a R_b),
\]
with Gram matrix $D_x = (\Tr(R_{e_i}R_{e_j}))_{i,j}$ in the basis $(e_i)$.  A short computation with
structure constants shows
\[
  D_x(i,j) = \sum_{q,r} x^{r}_{q,i}x^{q}_{r,j}.
\]
For $g\in \GL(V)$, the right multiplication operators transform by conjugation,
$R_{g\cdot a} = g R_a g^{-1}$, and therefore
\[
  D_{g\cdot x} = g^{-t}D_x g^{-1},\qquad \det D_{g\cdot x}=(\det g)^{-2}\det D_x.
\]

\begin{theorem}\label{thm:Killing-Leibniz}
  Assume $\kk$ has characteristic $0$ and $x\in \Leib(V)$.  Then $\det D_x\neq 0$, equivalently
  $\kappa_R$ is nondegenerate, if and only if $(V,\mu_x)$ is a semisimple Lie algebra.
\end{theorem}

\begin{proof}
  Assume first that $\kappa_R$ is nondegenerate.  As observed above, for any $a\in V$ the square
  $\mu_x(a,a)$ lies in the right annihilator, so $R_{\mu_x(a,a)} = 0$.  For any $b\in V$,
  \[
    \kappa_R\bigl(\mu_x(a,a),b\bigr) = \Tr\bigl(R_{\mu_x(a,a)}R_b\bigr) = 0,
  \]
  so every square $\mu_x(a,a)$ lies in the radical of $\kappa_R$.  Nondegeneracy therefore forces
  \[
    \mu_x(a,a) = 0 \qquad\text{ for all } a\in V.
  \]
  Since $\mu_x$ is bilinear and $\kk$ has characteristic 0, this implies skew-symmetry.  Indeed,
  \[
    0 = \mu_x(a + b , a + b)
    = \mu_x(a,b) + \mu_x(b,a)
  \]
  for all $a,b$, so $\mu_x(b,a) = -\mu_x(a,b)$.

  Thus $(V,\mu_x)$ is a Lie algebra.  On the Lie locus we have $L_a(b) = \mu_x(a,b)$ and
  $R_a(b) = \mu_x(b,a) = -\mu_x(a,b) = -L_a(b)$, so $R_a = -L_a$ and
  \[
    \kappa_R(a,b) = \Tr(R_a R_b) = \Tr(L_a L_b).
  \]
  The Killing form $\kappa_x$ of the Lie algebra $V_x=(V,\mu_x)$ is defined by
  \[
    \kappa_x(a,b) := \Tr(\mathrm{ad}(a)\mathrm{ad}(b)),
    \qquad
    \mathrm{ad}(a) = L_a-R_a.
  \]
  On the Lie locus we have $\mathrm{ad}(a) = 2L_a$, so
  \[
    \kappa_x(a,b) = \Tr(4L_aL_b) = 4 \kappa_R(a,b),
  \]
  and $\kappa_R = \tfrac14\kappa_x$.  Cartan established that the Killing form is nondegenerate if
  and only if the Lie algebra is semisimple.  Thus nondegeneracy of $\kappa_R$ forces $(V,\mu_x)$
  to be a semisimple Lie algebra.

  Conversely, if $V_x$ is a semisimple Lie algebra then its Killing form $\kappa_x$ is
  nondegenerate, and the preceding relation shows that $\kappa_R = \tfrac14\kappa_x$ is also
  nondegenerate.  Hence $\det D_x\neq 0$.
\end{proof}

Define
\[
  U_{\Leib} := \{ x\in \Leib(V) : \det D_x\neq 0 \}.
\]
By Theorem~\ref{thm:Killing-Leibniz}, $U_{\Leib}$ is exactly the locus of semisimple Lie brackets
on $V$.  It is $\GL(V)$-stable and contained in the Lie subvariety $\Lie(V)\subset \Leib(V)$, and
it can only be Zariski-dense in $\Leib(V)$ if $\Leib(V)$ coincides with the Lie locus.

\subsection{Rigidity on the semisimple Lie stratum}

On the Lie locus, Leibniz cohomology with coefficients in the adjoint module identifies with
Chevalley--Eilenberg cohomology.  More precisely, if $\mu_x$ is skew-symmetric and satisfies
Jacobi, then the Leibniz coboundary $\delta_x$ on skew-symmetric cochains coincides with the
Chevalley--Eilenberg differential $d_{\mathrm{CE}}$.  In particular, for $x\in \Lie(V)$ we have
\[
  Z^2_{\mathrm{Leib}}(V,V)_{\mathrm{skew}}\cong Z^2_{\mathrm{CE}}(V_x,V_x),
  \quad
  B^2_{\mathrm{Leib}}(V,V)_{\mathrm{skew}}\cong B^2_{\mathrm{CE}}(V_x,V_x),
\]
and the corresponding quotient identifies with $H^2_{\mathrm{CE}}(V_x,V_x)$.

If $V_x$ is semisimple, Whitehead’s Lemma (\cite[Chap.~XII]{CartanEilenberg1956}) gives
\[
  H^2_{\mathrm{CE}}(V_x,V_x) = 0.
\]
Thus for $x\in U_{\Leib}$ the stack tangent
\[
  T_{[x]}[\Leib(V)/\GL(V)] \cong H^2_{\mathrm{Leib}}(V,V)
\]
vanishes, and semisimple Lie points are formally rigid as Leibniz brackets.

\subsection{The Lie quotient and rigidity}\label{sec:Lie-quotient-rigidity}

Inside $\Leib(V)$, impose skew-symmetry via the linear ideal
\[
  I_{\mathrm{skew}} := \langle x^\ell_{i,j} + x^\ell_{j,i}\rangle
  \subset \kk[x_{i,j}^\ell]
\]
and define
\[
  \Lie(V) := \mathrm{V}(I_{\mathrm{skew}}) \cap \Leib(V).
\]
On $\mathrm{V}(I_{\mathrm{skew}})$, the right Leibniz identity is equivalent to the Jacobi
identity, so $\Lie(V)$ parameterizes Lie brackets on $V$.

For $x\in \Lie(V)$, let $V_x = (V,\mu_x)$ and consider the operators $L_a,R_a$ as before.  On the
Lie locus we have $R_a = -L_a$ and the Killing form $\kappa_R(a,b) = \Tr(R_aR_b)$ coincides with
$\tfrac14$ of the usual Killing form.  Let $D_x = (\Tr(R_{e_i}R_{e_j}))_{i,j}$ be the associated
Gram matrix and set
\[
  U_{\Lie} := \{ x\in \Lie(V) : \det D_x \neq 0 \}.
\]
By Theorem~\ref{thm:Killing-Leibniz}, $U_{\Lie}$ is exactly the locus of semisimple Lie brackets on
$V$.

\begin{proposition}\label{prop:tangent-Lie}
  For $x\in \Lie(V)$, we have canonical isomorphisms
  \[
    T_x\Lie(V) \cong Z^2_{\mathrm{CE}}(V_x,V_x),\qquad
    T_x(\GL(V) \cdot x) \cong B^2_{\mathrm{CE}}(V_x,V_x),
  \]
  and
  \[
    T_{[x]}[\Lie(V)/\GL(V)] \cong H^2_{\mathrm{CE}}(V_x,V_x).
  \]
\end{proposition}

\begin{proof}
  By Proposition~\ref{prop:Leib-tangent=Z2} we have
  \[
    T_x\Leib(V)\cong Z^2_{\mathrm{Leib}}(V,V)
  \]
  and $T_x(\GL(V) \cdot x)\cong B^2_{\mathrm{Leib}}(V,V)$, with the stack tangent identified with
  $H^2_{\mathrm{Leib}}(V,V)$.  Imposing skew-symmetry corresponds to intersecting with the subspace
  of skew-symmetric cochains.  On this subspace, the Leibniz differential $\delta_x$ coincides with
  the Chevalley--Eilenberg differential $d_{\mathrm{CE}}$ for the Lie algebra $V_x$. This is an
  immediate consequence of the definitions once one uses skew-symmetry and Jacobi to rewrite the
  Leibniz identity in Lie form.  Hence
  \[
    T_x\Lie(V)\cong Z^2_{\mathrm{Leib}}(V,V)_{\mathrm{skew}}
    \cong Z^2_{\mathrm{CE}}(V_x,V_x),
  \]
  and similarly for coboundaries and cohomology.  The identification of the orbit tangent with
  $B^2_{\mathrm{CE}}$ follows by restricting the argument of Proposition~\ref{prop:Leib-tangent=Z2}
  to skew-symmetric directions, and the stack tangent is again the quotient of cocycles by
  coboundaries.
\end{proof}

\begin{theorem}\label{thm:rigid-ULie}
  Assume $\kk$ has characteristic $0$.  For every point $x\in U_{\Lie}$, we have
  \[
    T_x\Lie(V)=T_x(\GL(V) \cdot x)
    \quad\text{and}\quad
    T_{[x]}[\Lie(V)/\GL(V)] = 0.
  \]
  Equivalently, the $\GL(V)$-orbit of $x$ is Zariski open in $\Lie(V)$ near $x$ and the
  corresponding point in the coarse moduli space $\Lie(V)\sslash\GL(V)$ is isolated.
\end{theorem}

\begin{proof}
  If $x\in U_{\Lie}$, then $V_x$ is semisimple.  Whitehead’s
  Lemma~\cite[Chap.~XII]{CartanEilenberg1956} yields
  \[
    H^2_{\mathrm{CE}}(V_x,V_x) = 0.
  \]
  Thus $Z^2_{\mathrm{CE}}(V_x,V_x) = B^2_{\mathrm{CE}}(V_x,V_x)$.
  Proposition~\ref{prop:tangent-Lie} then gives
  \[
    T_x\Lie(V)\cong Z^2_{\mathrm{CE}}(V_x,V_x)
    \cong B^2_{\mathrm{CE}}(V_x,V_x)
    \cong T_x(\GL(V) \cdot x),
  \]
  so the $\GL(V)$-orbit exhausts the tangent space and is therefore open at $x$.  The quotient
  tangent
  \[
    T_{[x]}[\Lie(V)/\GL(V)] \cong H^2_{\mathrm{CE}}(V_x,V_x)
  \]
  vanishes, which shows that the corresponding point in the coarse moduli space has no nontrivial
  first-order deformations and is isolated.
\end{proof}

\begin{corollary}\label{cor:det-open-rigid}
  The open subset $U_{\Lie}$ is a finite union of open $\GL(V)$-orbits, one for each semisimple Lie
  isomorphism class on $V$.  Points of $U_{\Lie}$ are isolated in the coarse moduli space
  $\Lie(V)\sslash \GL(V)$.
\end{corollary}

\begin{proof}
  By Theorem~\ref{thm:Killing-Leibniz}, $U_{\Lie}$ is exactly the locus of semisimple Lie brackets.
  Theorem~\ref{thm:rigid-ULie} shows that each such point has an open $\GL(V)$-orbit and vanishing
  stack tangent, so each semisimple isomorphism class determines an open orbit that is isolated in
  the coarse moduli space.  Distinct semisimple Lie algebras are not isomorphic, so these orbits
  are pairwise disjoint.
\end{proof}

For a semisimple Lie algebra $V_x$ we have $\Der(V_x) = \mathrm{ad}(V_x)$ and
$\dim\Der(V_x) = \dim V_x=n$.  The tangent space to the $\GL(V)$-orbit at $x$ has dimension
\[
  \dim T_x(\GL(V) \cdot x)
  = \dim\mathfrak{gl}(V) - \dim\Aut(V_x)
  = n^2 - n.
\]
By Theorem~\ref{thm:rigid-ULie}, $\dim T_x\Lie(V) = n^2 - n$ as well, so $\Lie(V)$ is smooth of
dimension $n^2 - n$ at $x$ and the orbit is open.

\begin{corollary}\label{cor:count-Lie-orbits}
  Assume $\kk$ is algebraically closed of characteristic $0$.  The number of open $\GL(V)$-orbits
  in the semisimple locus $U_{\Lie}\subset \Lie(V)$ is equal to the number of ways to write $n$ as
  a sum of dimensions of simple Lie algebras:
  \[
    N_{\mathrm{Lie}}(n)
    := \#\left\{ (\mathfrak{g}_1,\ldots,\mathfrak{g}_k) : \mathfrak{g}_i \text{ simple},\ \sum_{i=1}^k \dim \mathfrak{g}_i = n\right\} \big/ {\sim},
  \]
  where $(\mathfrak{g}_1,\ldots,\mathfrak{g}_k) \sim (\mathfrak{g}'_1,\ldots,\mathfrak{g}'_{k'})$
  if they represent the same isomorphism class up to reordering.
\end{corollary}

\begin{proof}
  Over an algebraically closed field of characteristic $0$, every semisimple Lie algebra $V_x$ of
  dimension $n$ decomposes uniquely as a direct sum of simple Lie algebras
  \[
    V_x\cong \bigoplus_{i=1}^k \mathfrak g_i.
  \]
  Isomorphism classes of semisimple Lie algebras of dimension $n$ are in bijection with unordered
  tuples $(\mathfrak g_1,\dots,\mathfrak g_k)$ of simple factors whose dimensions sum to $n$.  By
  Corollary~\ref{cor:det-open-rigid}, each such isomorphism class corresponds to a unique open
  $\GL(V)$-orbit in $U_{\Lie}$, and distinct isomorphism classes yield disjoint orbits.  This gives
  the stated count.
\end{proof}

\begin{remark}
  The number of open orbits in the Lie case for small dimensions is tabulated in \cite[A178930]{OEIS}:
  \[
    N_{\mathrm{Lie}}(1), \ldots, N_{\mathrm{Lie}}(20) = 0,0,1,0,0,1,0,1,1,1,1,1,1,2,2,2,2,3,2,3.
  \]
  The initial values are governed by the small list of simple Lie dimensions (e.g.\ $3$ for
  $\mathfrak{sl}_2$, $8$ for $\mathfrak{sl}_3$, $10$ for $\mathfrak{so}_5$, $14$ for
  $\mathfrak{g}_2$, $15$ for $\mathfrak{sl}_4$) and their finite sums.  For instance,
  $N_{\mathrm{Lie}}(14) = 2$ because one has either $\mathfrak{g}_2$ or
  $\mathfrak{sl}_3\oplus \mathfrak{sl}_2^{\oplus 2}$, and $N_{\mathrm{Lie}}(15)=2$ because one has
  either $\mathfrak{sl}_4$ or $\mathfrak{sl}_2^{\oplus 5}$.
\end{remark}


\end{document}